\documentclass[a4paper,12pt]{article}
\usepackage{graphicx} 
\usepackage{hyperref}
\usepackage{amsmath,amssymb}
\usepackage{float}
\usepackage{amsfonts}
\usepackage{braket}
\usepackage{xspace}
\usepackage{tikz-cd}
\tikzset{widthone/.style={draw, minimum width=0.6cm, fill=white, minimum height=16pt, inner sep=-10pt}}
\tikzset{widthtwo/.style={draw, minimum width=1.6cm, fill=white, minimum height=16pt, inner sep=-10pt}}
\tikzset{widththree/.style={draw, minimum width=2.2cm, fill=white}}
\tikzset{label/.style={font=\scriptsize}}
\usepackage{amsthm}
\theoremstyle{definition}
\newtheorem{definition}{Definition}[section]
\newtheorem{example}[definition]{Example}

\theoremstyle{satz}
\newtheorem{satz}[definition]{Proposition}

\theoremstyle{theorem}
\newtheorem{theorem}[definition]{Theorem}

\theoremstyle{corollary}

\theoremstyle{lemma}

\newcounter{remark}[section]

\usepackage{graphicx}
\usepackage{verbatim}
\tikzset{dot/.style={draw, fill, circle, inner sep=0pt, minimum width=5pt}}

\newcommand\ignore[1]{}

\newcommand\vc[1]{\begin{tabular}{ccc}#1\end{tabular}}

\def\C{\ensuremath{\mathcal{C}}}

\def\N{\mathbb{N}}

\title{On Structures in Arrow Categories}
\author{\small\vc{\vc{Paulina L. A. Goedicke\\Institute for Theoretical Physics\\University of Cologne\\\texttt{goedicke@thp.uni-koeln.de}} & \vc{Jamie Vicary\\Department of Computer Science\\University of Cambridge\\\texttt{jamie.vicary@cl.cam.ac.uk}}}
}
\date{}

\begin{document}

\maketitle
\begin{abstract}\noindent
    In this article we investigate which categorical structures of a category $\C$ are inherited by its arrow category. In particular, we show that a monoidal equivalence between two categories gives rise to a monoidal equivalence between their arrow categories. Moreover, we examine under which circumstances an arrow category is rigid and pivotal. Finally, we derive what the (co)algebra, bialgebra and Hopf algebra objects are in an arrow category. 
\end{abstract}
\tableofcontents
\section{Introduction}
Given a category $\C$, the category of arrows of $\C$ is a very fundamental concept in category theory~\cite{IntroCat}. This article is a revision of a chapter from the first author's master's thesis\footnote{The thesis was submitted in October 2021 at the Department of Mathematics at University of Hamburg.}, investigating which categorical structures an arrow category inherits from its underlying category $\C$. We start by discussing functors and natural transformations in Section~\ref{sec: functors} and show that an equivalence between two categories gives rise to an equivalence between their arrow categories. In Section~\ref{sec: monoidal product in arrow cats} we then extend this to braided monoidal categories and functors. In Section~\ref{sec: duals in arrow cats} we prove that the arrow category of a rigid monoidal category restricted to objects that are isomorphisms is also rigid. We further show that in that case, if the underlying category is in addition pivotal, its arrow category is also pivotal. In fact, it is a ribbon category. In Section~\ref{sec: monoids comonoids in arr cats} we then discuss (co)monoids, bialgebras, Frobenius structures and Hopf algebras in arrow categories. \\\\
There already have been definitions of monoidal products in arrow categories~\cite{white2018arrow} and the rest of our results also seem to be quite fundamental, however we still believe them to have some novelty. \\\\
We assume familiarity with the basic concepts of category theory, the graphical calculus and quantum groups and refer to~\cite{graphicalcalculus},~\cite{IntroCat} and~\cite{jamie} for an introduction. Here we will only briefly review the definition of an arrow category: 
\begin{definition}(\cite{IntroCat}, p.~24-25) \label{def. arrow category}
Let $\mathcal{C}$ be a category. The \textit{arrow category of} $\mathcal{C}$ Arr($\mathcal{C}$) is defined as follows:
\begin{itemize}
    \item \textit{objects} are triples ($A$, $B$, $h$) where $A$, $B$ $\in$ obj($\mathcal{C}$) and $h: A\longrightarrow B$.
    \item \textit{morphisms} $\phi: (A, B, h) \longrightarrow (A', B', h')$ are pairs ($\phi_{A}$, $\phi_{B}$) of morphims $\phi_{A}: A \longrightarrow A'$ and $\phi_{B}: B \longrightarrow B'$ in $\mathcal{C}$ such that the following diagram commutes:
    \[\begin{tikzcd}
    A \arrow[swap]{d}{h} \arrow{r}{\phi_A} &     \arrow{d}{h'} A' \\
    B \arrow{r}{\phi_B} & B'
\end{tikzcd}
\]
\end{itemize}
\end{definition}
\begin{example}\label{ex. arrow categoryy of Rel}The arrow category of \textbf{Mat}($\mathbb{N}$) has $b\times v$-matrices $\chi$ as objects, where $b$ and $v$ are natural numbers. A morphism $S: M \longrightarrow N$ between two matrices $M: v \longrightarrow b$ and $N: v^\prime \longrightarrow b^\prime$ is given by a pair of matrices $(S_{b}, S_{v}): (v, b) \longrightarrow (v', b')$ such that the following diagram commutes:
    \[\begin{tikzcd}
   b \arrow{r}{S_{b}} \arrow[swap]{d}{M} &     \arrow{d}{N} b'\\
   v \arrow{r}{S_{v}} & v'
\end{tikzcd}
\]
\end{example}

\section{Structures in Arrow Categories}
We will now discuss what kind of categorical structures one can define in arrow categories. 
\subsection{Functors in Arrow Categories}\label{sec: functors}
In the following we will show that a functor between two categories gives rise to a functor between their arrow categories and prove a similar statement for natural transformations.
\begin{satz}\label{prop. functor between arrow cat.}
 Let $\mathcal{C}$ and $\mathcal{D}$ be two categories and let $\mathcal{F}: \mathcal{C} \longrightarrow \mathcal{D}$ be a covariant functor between those categories. Then $\mathcal{F}$ gives rise to a covariant functor $\Tilde{\mathcal{F}}$ between Arr($\mathcal{C}$) and
Arr($\mathcal{D}$).
\end{satz}
\begin{proof}
Given a functor $\mathcal{F}: \mathcal{C} \longrightarrow \mathcal{D}$ that assigns to every object $A$ in $\mathcal{C}$ an object $\mathcal{F}(A)$ in $\mathcal{D}$ and to every morphism $f: A\longrightarrow B$ in $\mathcal{C}$ a morphism $\mathcal{F}(f): \mathcal{F}(A) \longrightarrow \mathcal{F}(B)$ in $\mathcal{D}$, we can define a functor $\Tilde{\mathcal{F}}$ that maps every object $f : A\longrightarrow B$ in Arr($\mathcal{C}$) to an object $\Tilde{\mathcal{F}}(f)=\mathcal{F}(f): \mathcal{F}(A) \longrightarrow \mathcal{F}(B)$ in Arr($\mathcal{D}$) and every morphism $(\phi, \psi): f \longrightarrow f'$ in Arr($\mathcal{C}$) to a morphism $\Tilde{\mathcal{F}}(\phi, \psi)=(\mathcal{F}(\phi), \mathcal{F}(\psi)): \mathcal{F}(f) \longrightarrow \mathcal{F}(f')$ in Arr($\mathcal{D}$). This is valid because the diagram
\[\begin{tikzcd}
    \mathcal{F}(A) \arrow{r}{\mathcal{F}(\phi)} \arrow[swap]{d}{\mathcal{F}(f)} &  \arrow{d}{ \mathcal{F}(f')} \mathcal{F}(A') \\
    \mathcal{F}(B) \arrow{r}{\mathcal{F}(\psi)} & \mathcal{F}(B')
\end{tikzcd}
\]
commutes due to functoriality of $\mathcal{F}$. It is easy to show that $\tilde{\mathcal{F}}$ preserves composition and the identity morphism in Arr($\mathcal{C}$).
\end{proof}\noindent
Similarly, one can prove the following statement:
\begin{satz}\label{rem. contravariant functor arrow cat}
  A contravariant functor $\mathcal{F}: \mathcal{C} \longrightarrow \mathcal{D}$ gives rise to a contravariant functor $\Tilde{\mathcal{F}}: \text{Arr}(\mathcal{C}) \longrightarrow \text{Arr}(\mathcal{D})$.  
\end{satz}

\begin{example} \label{ex. dualisation functor on Arr Mat} The contravariant functor $\mathcal{T}: \mathbf{Mat}(\N) \longrightarrow \mathbf{Mat}(\N)$ which maps each set $v$ to itself and each matrix $Mi: v \longrightarrow b$ to its transpose $M^{T}: b \longrightarrow v$ gives rise to a contravariant functor $\Tilde{\mathcal{T}}: \mathrm{Arr}(\mathbf{Mat}(\N)) \longrightarrow \mathrm{Arr}(\mathbf{Mat}(\N))$ which maps each matrix $M: v\longrightarrow b$ to its transpose $M^{T}: b \longrightarrow v$ and each morphism $(S_v, S_b): M \longrightarrow N$ to its transpose $(S_{v}^{T}, S_{b}^{T}): N^{T} \longrightarrow M^{T}$ such that the following diagram commutes:
\[\begin{tikzcd}
    b' \arrow{r}{S_{v}^{T}} \arrow[swap]{d}{N^{T}} &     \arrow{d}{M^{T}} b \\
    v' \arrow{r}{S^{T}_{v}} & v
\end{tikzcd}
\]
\end{example}
\begin{satz}\label{prop. natural trafo between arrow cats}
Let $\mathcal{C}$ and $\mathcal{D}$ be two categories and let $\mathcal{F}, \mathcal{G}: \mathcal{C} \longrightarrow \mathcal{D}$ be two covariant functors between those categories that induce the functors $\Tilde{\mathcal{F}}, \Tilde{\mathcal{G}}: \mathrm{Arr}(\mathcal{C}) \longrightarrow \mathrm{Arr}(\mathcal{D})$. Consider a natural transformation $\eta: \mathcal{F} \Longrightarrow \mathcal{G}$ between $\mathcal{F}$ and $\mathcal{G}$. Then $\eta$ induces a natural transformation $\Tilde{\eta}: \Tilde{\mathcal{F}} \Longrightarrow \Tilde{\mathcal{G}}$.
\end{satz} 
\begin{proof}
Let $\eta: \mathcal{F} \Longrightarrow \mathcal{G}$ be a natural transformation that assigns to every object $A$ in $\mathcal{C}$ a morphism $\eta_{A}: \mathcal{F}(A) \longrightarrow \mathcal{G}(A)$ such that for any morphism $f:A \longrightarrow B$ in $\mathcal{C}$ the following diagram (naturality condition) commutes:
\[\begin{tikzcd}
    \mathcal{F}(A) \arrow{r}{\eta_{A}} \arrow[swap]{d}{\mathcal{F}(f)} &  \arrow{d}{ \mathcal{G}(f)} \mathcal{G}(A) \\
    \mathcal{F}(B) \arrow{r}{\eta_{B}} & \mathcal{G}(B)
\end{tikzcd}
\]
Using this, once can assign to every object $f: A \longrightarrow B$ in Arr($\mathcal{C}$) a morphism $\Tilde{\eta}_{f}=(\eta_{A}, \eta_{B}): \Tilde{\mathcal{F}}(f) \longrightarrow \Tilde{\mathcal{G}}(f)$, such that for any morphism \\
$(\phi, \psi): f \longrightarrow f'$ in Arr($\mathcal{C}$), where $f': A' \longrightarrow B'$, the following diagram (naturality condition in the arrow category) commutes:
\[\begin{tikzcd}[row sep=1cm, column sep=1cm, inner sep=2ex]
& \mathcal{F}(A) \arrow[swap]{dl}{\mathcal{F}(\phi)}\arrow{rr}{\eta_{A}} \arrow[dotted]{dd}[yshift=-0.5cm]{\mathcal{F}(f)} & & \mathcal{G}(A) \arrow[swap]{dl}{\mathcal{G}(\phi)} \arrow{dd}{\mathcal{G}(f)} \\
\mathcal{F}(A') \arrow [crossing over]{rr}[xshift= 0.5cm]{\eta_{A'}} \arrow{dd} {\mathcal{F}(f')} & & \mathcal{G}(A') \arrow[crossing over]{dd} [yshift=0.5cm]{\mathcal{G}(f')}\\
& \mathcal{F}(B) \arrow[dotted]{dl}{\mathcal{F}(\psi)} \arrow[dotted]{rr}[xshift= -0.5cm]{\eta_{B}} & & \mathcal{G}(B) \arrow{dl}{\mathcal{G}(\psi)} \\
\mathcal{F}(B') \arrow[swap]{rr}{\eta_{B'}} & & \mathcal{G}(B')\\
\end{tikzcd}
\]
Here the the top, the back, the front and the bottom face commute due to naturality of $\eta$ and the two side faces commute by definition. Hence the whole diagram commutes and we have defined a natural transformation $\Tilde{\eta}: \Tilde{\mathcal{F}} \Longrightarrow \Tilde{\mathcal{G}} $.
\end{proof}\noindent
It is straightforward to verify that the following has to hold:
\begin{satz}\label{rem. natural isomoprhism between arrow cats}
 If $\eta: \mathcal{F} \Longrightarrow \mathcal{G}$ is a natural isomorphism, so is $\Tilde{\eta}: \Tilde{\mathcal{F}} \Longrightarrow \Tilde{\mathcal{G}} $.   
\end{satz}

\begin{theorem}\label{thrm. equivalence arrow categories}
Let $\mathcal{C}$ and $\mathcal{D}$ be two equivalent categories, i. e. there exist a pair of functors $\mathcal{F}: \mathcal{C} \longrightarrow \mathcal{D}$ and $\mathcal{G}: \mathcal{D} \longrightarrow \mathcal{C}$ and natural isomorphisms $\mathcal{F} \circ \mathcal{G} \cong \mathrm{id}_{\mathcal{D}}$ and $\mathcal{G} \circ \mathcal{F} \cong \mathrm{id}_{\mathcal{C}}$. Then Arr($\mathcal{C}$) and Arr($\mathcal{D}$) are also equivalent.
\end{theorem}
\begin{proof}
By Proposition~\ref{prop. functor between arrow cat.} the functors $\mathcal{F}: \mathcal{C} \longrightarrow \mathcal{D}$ and $\mathcal{G}: \mathcal{D} \longrightarrow \mathcal{C}$ give rise to functors $\Tilde{\mathcal{F}}: \mathrm{Arr}(\mathcal{C}) \longrightarrow \mathrm{Arr}(\mathcal{D})$ and $\Tilde{\mathcal{G}}: \mathrm{Arr}(\mathcal{D}) \longrightarrow \mathrm{Arr}(\mathcal{C})$. From Proposition~\ref{rem. natural isomoprhism between arrow cats} we know that the natural isomorphisms $\mathcal{F} \circ \mathcal{G} \cong \mathrm{id}_{\mathcal{D}}$ and $\mathcal{G} \circ \mathcal{F} \cong \mathrm{id}_{\mathcal{C}}$ give rise to natural isomorphisms $\Tilde{\mathcal{F}} \circ \Tilde{\mathcal{G}} \cong \mathrm{id}_{\mathrm{Arr}(\mathcal{D})}$ and $\Tilde{\mathcal{G}} \circ \Tilde{\mathcal{F}} \cong \mathrm{id}_{\mathrm{Arr}(\mathcal{C})}$. Hence we have an equivalence.
\end{proof}
With that, we can define a dagger structure on Arr($\mathcal{C}$), using the dagger structure in $\mathcal{C}$:
\begin{satz}\label{prop. dagger functor}
Let $\mathcal{C}$ be a dagger category and let $\mathcal{C}_{uni}$ be the subcategory where all morphisms are unitary, i. e. all $f \in \mathrm{Hom}_{\C}$ are invertible with $f^{-1} = f^{\dagger}$. Then $\mathrm{Arr}(\mathcal{C}_{uni})$ is a dagger category.
\end{satz}
\begin{proof}
Let $\mathcal{C}$ be a dagger category, i. e. there exists an involutive contravariant functor $\dagger: \mathcal{C} \longrightarrow \mathcal{C}$ such that
\begin{align*}
    (g\circ f)^{\dagger}= f^{\dagger} \circ g^{\dagger}\\
    \mathrm{id}^{\dagger}= \mathrm{id} \\
    (f^{\dagger})^{\dagger}=f.
\end{align*}
Restricting to the subcategory where all morphisms are unitary, i. e. we have for each morphism $\psi: A \longrightarrow A'$ in $\mathcal{C}$:
\begin{align}
    \psi \circ \psi^{\dagger}= \mathrm{id}_{A'} \ \ \ \text{and}\\
     \psi^{\dagger} \circ \psi= \mathrm{id}_{A}
\end{align}
and taking the arrow category of $\C$, we can construct a functor $\mathrm{Arr}^{\dagger}:\mathrm{Arr}(\mathcal{C}_{uni}) \longrightarrow \mathrm{Arr}(\mathcal{C}_{uni})$ which sends each object $f:A \longrightarrow B$ in $\mathrm{Arr}(\mathcal{C}_{uni})$ to itself and each morphism 
\[\begin{tikzcd}
    A \arrow{r}{\psi} \arrow[swap]{d}{f} & \arrow{d}{g} A' \\
   B \arrow{r}{\psi'} & B'
\end{tikzcd}
\]
to its adjoint $(\psi, \psi')^{\dagger}= (\psi^{\dagger}, \psi'^{\dagger})$:
\[\begin{tikzcd}
    A' \arrow{r}{\psi^{\dagger}} \arrow[swap]{d}{f} & \arrow{d}{g} A \\
   B' \arrow{r}{\psi'^{\dagger}} & B
\end{tikzcd}
\]
The above diagram commutes because we have $\psi' \circ f = g \circ \psi$ per definition and $\psi \circ \psi^{\dagger} = \mathrm{id}_{B'}$.
It is easy to verify that this construction fulfils the requirements for a dagger functor.
\end{proof}

\subsection{Monoidal Products and Braidings in Arrow Categories}\label{sec: monoidal product in arrow cats}
In this section we will define a braiding in arrow categories and demonstrate that a monoidal functor between two categories induces a monoidal functor between their arrow categories. Moreover, we will show that, given a symmetric monoidal category $\mathcal{C}$, its arrow category is also symmetric.\\\\
The following statement can be found in a similar manner in~\cite{white2018arrow}. We will still give its proof. 
\begin{satz}\label{prop. monoidal product in arrow cats}
If $\mathcal{C}$ is a monoidal category, we can use the monoidal product in $\mathcal{C}$ to define a pointwise monoidal product in Arr($\mathcal{C}$). 
\end{satz}
\begin{proof}
   Let $(\mathcal{C}, \otimes_{\mathcal{C}}, \mathbb{I}_{\mathcal{C}}, \lambda_{\mathcal{C}}, \rho_{\mathcal{C}})$ be a monoidal category and Arr($\mathcal{C}$) its arrow category. A monoidal structure in Arr($\mathcal{C}$) can be defined as follows:
\begin{itemize}
\item {On objects we have
    \begin{align*}
       (f: A_{1} \longrightarrow B_{1}) \otimes_{\mathrm{Arr}(\mathcal{C})} (g: A_{2} \longrightarrow B_{2}) := f \otimes_{\mathcal{C}} g : A_{1} \otimes_{\mathcal{C}} A_{2} \longrightarrow B_{1} \otimes_{\mathcal{C}} B_{2}
    \end{align*}
where $f:A_{1} \longrightarrow B_{1}$ and $g:A_{2} \longrightarrow B_{2}$ are objects in Arr($\mathcal{C}$).}
\item {On morphisms $(\phi_{A_{1}}, \phi_{B_{1}})$ and $(\phi_{A_{2}}, \phi_{B_{2}})$, where $\phi_{A_{i}}: A_{i} \longrightarrow A_{i}'$ and  $\phi_{B_{i}}: B_{i} \longrightarrow B_{i}'$ for $i=1,2$, we have
\begin{align}
    (\phi_{A_{1}}, \phi_{B_{1}})\otimes_{\mathrm{Arr}(\mathcal{C})} (\phi_{A_{2}}, \phi_{B_{2}}) := (\phi_{A_{1}}\otimes_{\mathcal{C}}\phi_{A_{2}}, \phi_{B_{1}}\otimes_{\mathcal{C}}\phi_{B_{2}})
\end{align}
such that the following diagram commutes:
\[\begin{tikzcd}
    A_{1}\otimes_{\mathcal{C}} A_{2} \arrow{r}{f \otimes_{\mathcal{C}} g} \arrow[swap]{d}{\phi_{A_{1}}\otimes_{\mathcal{C}}\phi_{A_{2}}} &     \arrow{d}{\phi_{B_{1}'}\otimes_{\mathcal{C}}\phi_{B_{2}'}}  B_{1} \otimes_{\mathcal{C}} B_{2} \\
   A_{1}'\otimes_{\mathcal{C}} A_{2}' \arrow{r}{f' \otimes_{\mathcal{C}} g'} &  B_{1}' \otimes_{\mathcal{C}} B_{2}' 
\end{tikzcd}
\]}
\item{The unit object is given by the identity morphism on the monoidal unit in $\mathcal{C}$:
\begin{align}
    \mathrm{id}_{\mathbb{I}_{\mathcal{C}}}:\mathbb{I}_{\mathcal{C}}\longrightarrow \mathbb{I}_{\mathcal{C}}.
\end{align}}
\item{Since left- and right-unitors are natural isomorphisms one can use Propisition~\ref{rem. natural isomoprhism between arrow cats} to define
\begin{align}
    \lambda_{f}: \mathrm{id}_{\mathbb{I}_{\mathcal{C}}}\otimes_{\mathcal{C}} f  \longrightarrow f
\end{align}
and 
\begin{align}
    \rho_{f}: f \otimes_{\mathcal{C}}\mathrm{id}_{\mathbf{I}_{\mathcal{C}}} \longrightarrow f
\end{align}
where $f: A \longrightarrow B$, i. e. the following diagrams commute:
\[\begin{tikzcd}
    A\otimes_{\mathcal{C}} \mathbb{I}_{\mathcal{C}} \arrow{r}{\rho_{A}} \arrow[swap]{d}{f\otimes_{\mathcal{C}} \mathrm{id}_{\mathbb{I}_{\mathcal{C}}}} &     \arrow{d}{f} A \\
    B \otimes_{\mathcal{C}} \mathbb{I}_{\mathcal{C}} \arrow{r}{\rho_{B}} & B
\end{tikzcd}\]

\[\begin{tikzcd}
    \mathbb{I}_{\mathcal{C}} \otimes_{\mathcal{C}} A \arrow{r}{\lambda_{A}} \arrow[swap]{d}{\mathrm{id}_{\mathbb{I}_{\mathcal{C}}}\otimes_{\mathcal{C}} f} &  \arrow{d}{ f} A \\
\mathbb{I}_{\mathcal{C}}\otimes_{\mathcal{C}} B \arrow{r}{\lambda_{B}} & B
\end{tikzcd}
\]}
\item{Finally, the associator is a natural isomorphism
\begin{align}
    \alpha: (f_{1}\otimes_{\mathcal{C}} f_{2})\otimes_{\mathcal{C}} f_{3} \longrightarrow f_{1}\otimes_{\mathcal{C}} (f_{2}\otimes_{\mathcal{C}} f_{3})
\end{align}
where $f_{i}: A_{i} \longrightarrow B_{i}$ for $i=1,2,3$, such that the diagram
\[\begin{tikzcd}
    (A_{1} \otimes_{\mathcal{C}} A_{2})\otimes_{\mathcal{C}} A_{3} \arrow{r}{\alpha_{A_{1}, A_{2}, A_{3}}} \arrow[swap]{d}{(f_{1}\otimes_{\mathcal{C}} f_{2})\otimes_{\mathcal{C}} f_{3}} &  \arrow{d}{f_{1}\otimes_{\mathcal{C}} (f_{2}\otimes_{\mathcal{C}} f_{3})}  A_{1} \otimes_{\mathcal{C}} (A_{2}\otimes_{\mathcal{C}} A_{3})\\
     (B_{1} \otimes_{\mathcal{C}} B_{2})\otimes_{\mathcal{C}} B_{3} \arrow{r}{\alpha_{B_{1}, B_{2}, B_{3}}} &  B_{1} \otimes_{\mathcal{C}} (B_{2}\otimes_{\mathcal{C}} B_{3})
\end{tikzcd}
\]
commutes.}
\end{itemize}
We still have to show that the pentagon and the triangle axiom are satisfied. The proof for this can be found in the appendix.
\end{proof}
\begin{example}\label{ex. monoidal product Mat(N)}
The monoidal product in Arr(\textbf{Mat}($\mathbb{N}$)) is defined on objects $M: v_1 \longrightarrow b_1$, $N: v_2 \longrightarrow b_2$ via the Kronecker product of matrices: $M\otimes N : v_1 \cdot v_2 \to b_1 \cdot b_2$. On morphisms the monoidal product is defined as follows: $(S_{v_1}, S_{b_1}) \otimes (S_{v_2}, S_{b_2}) = (S_{v_1}\otimes S_{v_2}, S_{b_1} \otimes S_{b_2} )$, i. e. we have the commutative diagram:
  \[\begin{tikzcd}
   v_{1} \cdot v_{2} \arrow{r}{S_{v_{1}} \otimes S_{v_{2}}} \arrow[swap]{d}{M \otimes N} &  \arrow{d}{ M' \otimes N'}  v_{1}' \cdot v_{2}' \\
    b_{1} \cdot b_{2} \arrow{r}{ S_{b_{1}} \otimes S_{b_{2}}} & b_{1}' \cdot b_{2}'
\end{tikzcd} 
\]
The tensor unit is given by the $1\times1$-matrix:
\begin{align}
\mathbb{I}: 1 \longrightarrow 1.
\end{align}
\end{example}
In a similar way one can define a braiding in Arr($\mathcal{C}$) using the braiding in $\mathcal{C}$.
\begin{satz}\label{prop. braiding in arrow categories}
If $\mathcal{C}$ is a braided monoidal category with braiding $\sigma_{A,B}: A \otimes_{\mathcal{C}} B \longrightarrow B \otimes_{\mathcal{C}} A$, then Arr($\mathcal{C})$ has a braiding given by $\sigma_{f,g}=(\sigma_{A,C}, \sigma_{B,D}): f \otimes_{\mathcal{C}}g \longrightarrow g \otimes_{\mathcal{C}} f$ for $f: A \longrightarrow B$ and $g: C \longrightarrow D$, i. e. we have the following commutative diagram:
\[\begin{tikzcd}
   A \otimes_{\mathcal{C}} C \arrow{r}{\sigma_{A,C}} \arrow[swap]{d}{f\otimes_{\mathcal{C}} g} &  \arrow{d}{ g\otimes_{\mathcal{C}}f}  C\otimes_{\mathcal{C}} A  \\
    B \otimes_{\mathcal{C}} D \arrow{r}{\sigma_{B,D}} & D \otimes_{\mathcal{C}} B
\end{tikzcd} 
\]
\end{satz}
\begin{proof}
We can use Prop. \ref{rem. natural isomoprhism between arrow cats} to define a natural isomorphism  $\sigma_{f,g}=(\sigma_{A,C}, \sigma_{B,D})$ on Arr($\mathcal{C})$ using the natural isomorphisms $\sigma_{A,C}$ and $\sigma_{B,D}$. In order to define a braiding,  $\sigma_{f,g}=(\sigma_{A,C}, \sigma_{B,D})$ has to satisfy the hexagon identities. Here we will only discuss one of them as the other one can be proved analogously. In the diagram below the top and the bottom face commute because  $\sigma_{A,C}$ resp. $\sigma_{B,D}$ is a braiding in $\mathcal{C}$. The left side face and the right square in the front commute, because the braiding in $\mathcal{C}$ is natural and because of the definition of the monoidal product in Arr($\mathcal{C}$). The left square in the front commutes because the associator is natural in $\mathcal{C}$. Similarly, the back face commutes because of naturality of the associator in $\mathcal{C}$. Finally, the right side face commutes due to naturality of the braiding in $\mathcal{C}$ and hence the whole diagram commutes.\footnote{We have again abbreviated the labels of the monoidal products and the associator in order to make the diagram more readable.}
\[
\def\stimes{{\otimes}}
\hspace{-3.5cm}
\small
 \begin{tikzcd}[ampersand replacement=\&, row sep=0.7cm, column sep=0.7cm]
\& (A \otimes D) \otimes X
\arrow[swap]{dl}{\sigma_{A, D} \otimes \mathrm{id}_{X}}\arrow{rrrrrr}[xshift=-0.5cm]{\alpha} \arrow[dotted]{dd}[yshift=-0.5cm]{(f \otimes g) \otimes h} \& \& \& \& \& \& A \otimes (D \otimes X)  \arrow[swap]{dl}{\sigma_{A, D, X}} \arrow{dd}{f\otimes(g \otimes h)} \\
(D \otimes A) \otimes X \arrow [crossing over]{rr}[xshift= 1cm]{\alpha} \arrow{dd} [yshift= 0.5cm] {(g \otimes f) \otimes h} \& \&  D\otimes(A \otimes X) \arrow {dd}[yshift= 0.5 cm]{g \otimes (f \otimes h)} \arrow{rr}{\mathrm{id}_{D} \otimes \sigma_{A,X}} \& \& \arrow{dd}[yshift= 0.5cm]{g \otimes (h \otimes f)} D \otimes ( X \otimes A)  \& \& \arrow{ll}{\alpha} (D \otimes X)\otimes A \arrow[crossing over]{dd}[yshift= 0.5cm]{(g \otimes h) \otimes f}\\
\& (B \otimes C) \otimes Y
\arrow[dotted]{dl}{\sigma_{B, C} \otimes \mathrm{id}_{Y}}\arrow[dotted]{rrrrrr}[xshift=-0.5cm]{\alpha}  \& \& \& \& \& \& B \otimes (C \otimes Y)  \arrow[swap]{dl}{\sigma_{B, C, Y}}  \\
(C \otimes B) \otimes Y \arrow [crossing over]{rr}[xshift= 0.5cm]{\alpha}  \& \& C\otimes (B \otimes Y)  \arrow{rr}{\mathrm{id}_{C} \otimes \sigma_{B,Y}} \& \&  C \otimes ( Y \otimes B) \& \& \arrow{ll}{\alpha} (C \otimes Y)\otimes B
\end{tikzcd}
\]

\end{proof}

\begin{example}
On Arr(\textbf{Mat}($\mathbb{N}$)) one can define a braiding via
\begin{align}
    \sigma_{M,N}= (\sigma_{v_{1}, v_{2}}, \sigma_{b_{1}, b_{2}}): M \otimes N \longrightarrow N \otimes M
\end{align}
where $M$ is a $v_{1}\cdot v_{2} \times b_{1} \cdot b_{2}$-matrix and $N$ is a $v_{2}\cdot v_{1} \times b_{2}\cdot b_{1}$-matrix.
\end{example}
\begin{theorem}\label{thrm. symmteric arrow categories}
If $\mathcal{C}$ is a symmetric monoidal category, then so is Arr($\mathcal{C}$).
\end{theorem}
\begin{proof}
If $\mathcal{C}$ is symmetric the braiding fulfils the following condition for arbitrary objects $A$, $C$ in $\mathcal{C}$
\begin{align}
    \sigma_{A,C} \circ \sigma_{C,A} = \mathrm{id}_{A \otimes_{\mathcal{C}}C}.
\end{align}
In Arr($\mathcal{C}$) symmetry would require that
\begin{align*}
    \sigma_{f,g}\circ \sigma_{g,f}=\mathrm{id}_{f\otimes_{Arr(\mathcal{C})}g}
\end{align*}
where $f$ and $g$ are morphisms in $\mathcal{C}$. In others words, the following diagram has to commute
\[\begin{tikzcd}[row sep=scriptsize, column sep=scriptsize]
A\otimes_{\mathcal{C}} D \arrow{dd}{f\otimes_{\mathcal{C}} g} \arrow{dr}{\sigma_{A,D} } \arrow{rr}{\mathrm{id}_{A\otimes_{\mathcal{C}} D}} & &  A\otimes_{\mathcal{C}} D \arrow{dd}{f\otimes_{\mathcal{C}} g} \arrow{dl}{\sigma_{A,D} } \\
&   D \otimes_{\mathcal{C}} A \arrow{dd}[yshift= 0.5 cm]{g\otimes_{\mathcal{C}} f} & \\
B\otimes_{\mathcal{C}} C \arrow{dr}{\sigma_{B,C} } \arrow[dotted]{rr}[xshift= -0.5cm]{\mathrm{id}_{B\otimes_{\mathcal{C}}C}} & & B\otimes_{\mathcal{C}} C \arrow{dl}{\sigma_{B,C}} \\
 & C\otimes_{\mathcal{C}} B  & 
\end{tikzcd} 
\]
The top and the bottom face commute because $\mathcal{C}$ is symmetric and the two side faces commute as the braiding in $\mathcal{C}$ is natural. The back face commutes due to the definition of the identity morphism in Arr($\mathcal{C}$) and hence the whole diagram commutes which proves the assumption.
\end{proof}
As it is possible to define a monoidal product in arrow categories, one can ask if a monoidal functor between two categories gives rise to a monoidal functor between their arrow categories. This is indeed the case, as the following proposition shows:
\begin{satz}\label{prop. monoidal functors in arrow categories}
Let $(\mathcal{C}, \otimes_{\mathcal{C}}, \mathbb{I}_{\mathcal{C}})$ and $(\mathcal{D}, \otimes_{\mathcal{D}}, \mathbb{I}_{\mathcal{D}})$ be two monoidal categories and $(\mathcal{F}, \mathcal{F}_{2}, \mathcal{F}_{0}) $ be a monoidal functor between them. Then $(\mathcal{F}, \mathcal{F}_{2}, \mathcal{F}_{0}) $ gives rise to a monoidal functor $(\Tilde{\mathcal{F}}, \Tilde{\mathcal{F}}_{2}, \Tilde{\mathcal{F}}_{0})$ between Arr($\mathcal{C}$) and Arr($\mathcal{D}$).
\end{satz}
\begin{proof}
We have already shown that a functor $\mathcal{F}$ between two categories $\mathcal{C}$ and $\mathcal{D}$ induces a functor  $\Tilde{\mathcal{F}}$ between their arrow categories and that a natural transformation $\eta: \mathcal{F} \Longrightarrow \mathcal{G}$ gives rise to a natural transformation $\Tilde{\eta}: \Tilde{\mathcal{F}} \Longrightarrow \Tilde{\mathcal{G}}$. Hence we can use the natural transformations $\mathcal{F}_{2}$ and $\mathcal{F}_{0}$ to define natural transformations 
\begin{align}
    \Tilde{\mathcal{F}}_{2}= (\mathcal{F}_{2, A, C}, \mathcal{F}_{2, B, D}): \Tilde{\mathcal{F}}(f) \otimes_{\mathcal{D}} \Tilde{\mathcal{F}}(g) \longrightarrow  \Tilde{\mathcal{F}}(f\otimes_{\mathcal{C}} g)
\end{align}
for $f: A \longrightarrow B$ and $g: C \longrightarrow D$
and 
\begin{align}
    \Tilde{\mathcal{F}}_{0}=(\mathcal{F}_{0}, \mathcal{F}_{0}): \mathrm{id}_{\mathbb{I}_{\mathcal{D}}} \longrightarrow  \Tilde{\mathcal{F}}(\mathrm{id}_{\mathbb{I}_{\mathcal{C}}}). 
\end{align}
In diagrams we have
\[\begin{tikzcd}
    \mathcal{F}(A)\otimes_{\mathcal{D}} \mathcal{F}(C) \arrow{r}{\mathcal{F}_{2, A, C}} \arrow[swap]{d}{\mathcal{F}(f)\otimes_{\mathcal{D}}\mathcal{F}(g)} &  \arrow{d}{\mathcal{F}(f \otimes_{\mathcal{C}} g)}  \mathcal{F}(A \otimes_{\mathcal{C}} C)   \\ 
    \mathcal{F}(B)\otimes_{\mathcal{D}} \mathcal{F}(D) \arrow{r}{\mathcal{F}_{2, B, D}} &   \mathcal{F}(B \otimes_{\mathcal{C}} D) 
\end{tikzcd} 
\]
and
\[\begin{tikzcd}
   \mathbb{I}_{\mathcal{D}} \arrow{r}{\mathcal{F}_{0}} \arrow[swap]{d}{\mathrm{id}_{\mathbb{I}_{\mathcal{D}}}} &  \arrow{d}{\mathcal{F}(\mathrm{id}_{\mathbb{I}_{\mathcal{C}}}} \mathcal{F}(\mathbb{I}_{\mathcal{C}})   \\
   \mathbb{I}_{\mathcal{D}} \arrow{r}{\mathcal{F}_{0}} &  \mathcal{F}(\mathbb{I}_{\mathcal{C}})
\end{tikzcd} 
\]
It is left to show that if $\mathcal{F}: \mathcal{C} \longrightarrow \mathcal{D}$ is monoidal, $\Tilde{\mathcal{F}}: \text{Arr}(\mathcal{C}) \longrightarrow \text{Arr}(\mathcal{D})$ preserves the monoidal product, i. e. $\Tilde{\mathcal{F}}_{2}$ and $\Tilde{\mathcal{F}_{0}}$ fulfil the associativity and the unitality conditions, respectively. The associativity condition in Arr($\mathcal{C}$) is given by the diagram on the next page. Here the top and the bottom face commute because $\mathcal{F}_{2}$ satisfies the associativity condition. The left face and the right square in the front commute because $\mathcal{F}_{2}$ is natural and because of the monoidal product in Arr($\mathcal{D}$). The right square in the front commutes because the associator in $\mathcal{C}$ is natural. Moreover, the left and the right face in the back commute as $\mathcal{F}_{2}$ is natural and the middle face in the back commutes due to naturality of the associator in $\mathcal{C}$.\footnote{Again we have abbreviated the labels of the monoidal products, associators and natural transformations $\mathcal{F}_{2}$ to make the diagram more readable.}
\[
\def\stimes{{\otimes}}
\hspace{-3.5cm}
\small
 \begin{tikzcd}[ampersand replacement=\&, row sep=0.4cm, column sep=0.3cm] 
\& (\mathcal{F}(A)\otimes\mathcal{F}(B)) \otimes \mathcal{F}(C)
\arrow[swap]{dl}{\mathcal{F}_{2}}\arrow{rrrr}{\alpha} \arrow[dotted]{dd}[yshift=-0.5cm]{(\mathcal{F}(f)\otimes \mathcal{F}(g)) \otimes \mathcal{F}(h)} \& \& \& \& \mathcal{F}(A)\otimes(\mathcal{F}(B) \otimes \mathcal{F}(C)) \arrow[swap]{dl}{\mathrm{id}_{\mathcal{F}(A)}\otimes \phi_{B,C}} \arrow{dd}{\mathcal{F}(f)\otimes(\mathcal{F}(g) \otimes \mathcal{F}(h))} \\
\mathcal{F}(A\otimes B) \otimes \mathcal{F}(C) \arrow [crossing over]{rr}[xshift= 1cm]{\mathcal{F}_{2}} \arrow{dd} {\mathcal{F}(f\otimes g) \otimes \mathcal{F}(h)} \&  \& \mathcal{F}((A\otimes B) \otimes C) \arrow [crossing over]{dd} [yshift=0.5cm]{\mathcal{F}((f\otimes g) \otimes h)} \arrow{rr}{\mathcal{F}(\alpha)} \& \&  \arrow{dd} [yshift= 0.5 cm]{\mathcal{F}(f\otimes( g \otimes h))} \mathcal{F}(A\otimes (B \otimes C)\\
\& (\mathcal{F}(V)\otimes \mathcal{F}(W)) \otimes \mathcal{F}(U)
\arrow[dotted]{dl}{\mathcal{F}_{2}}\arrow[dotted]{rrrr}{\alpha}  \& \& \& \& \mathcal{F}(V)\otimes (\mathcal{F}(W) \otimes \mathcal{F}(U)) \arrow[swap]{dl}{\mathrm{id}_{\mathcal{F}(V)}\otimes \phi_{W,U}} \\
\mathcal{F}(V\otimes W) \otimes \mathcal{F}(U) \arrow [crossing over]{rr}[xshift= 1cm]{\mathcal{F}_{2}}  \& \& \mathcal{F}((V\otimes W) \otimes U) \arrow{rr}{\mathcal{F}(\alpha)} \& \& \mathcal{F}(V\otimes(W \otimes U)\\
\end{tikzcd}
\]
The unitality condition is given by the following diagram:
\[\begin{tikzcd}[row sep=1cm, column sep=1cm, inner sep=2ex]
& \mathcal{F}(A)\otimes_{\mathcal{D}} \mathbb{I}_{\mathcal{D}} \arrow[swap]{dl}{\rho_{\mathcal{F}(A)}}\arrow{rr}{\mathrm{id}_{\mathcal{F}(A)}\otimes_{\mathcal{D}} \mathcal{F}_{0}} \arrow[dotted]{dd}[yshift=-0.5cm]{\mathcal{F}(f)\otimes_{\mathcal{D}} \mathrm{id}_{\mathbb{I}}} & & \mathcal{F}(A)\otimes_{\mathcal{D}} \mathcal{F}(\mathbb{I}_{\mathcal{C}}) \arrow[swap]{dl}{\mathcal{F}_{2, A \mathbb{I}}} \arrow[swap]{dd}[yshift=-0.5cm]{\mathcal{F}(f)\otimes_{\mathcal{D}} \mathcal{F}(\mathrm{id}_{\mathbb{I}})} \\
\mathcal{F}(A) \arrow [crossing over]{rr}[xshift= 1cm]{\mathcal{F}(\rho_{A})} \arrow{dd} {\mathcal{F}(f)} & & \mathcal{F}(A\otimes_{\mathcal{C}}\mathbb{I}_{\mathcal{C}}) \arrow[crossing over]{dd} [yshift=0.5cm]{\mathcal{F}(f\otimes_{\mathcal{C}} \mathrm{id}_{\mathbb{I}_{\mathcal{C}}})}\\
& \mathcal{F}(B)\otimes_{\mathcal{D}} \mathbb{I}_{\mathcal{D}} \arrow[dotted]{dl}{\rho_{\mathcal{F}(B)}} \arrow[dotted] {rr}[xshift= -1cm]{\mathrm{id}_{\mathcal{F}(B)}\otimes_{\mathcal{D}} \mathcal{F}_{0}} & & \mathcal{F}(B)\otimes_{\mathcal{D}} \mathcal{F}(\mathbb{I}_{\mathcal{C}}) \arrow{dl}{\mathcal{F}_{2, B, \mathbb{I}}} \\
\mathcal{F}(B) \arrow[swap]{rr}{\mathcal{F}(\rho_{B})} & & \mathcal{F}(B\otimes_{\mathcal{C}}\mathbb{I}_{\mathcal{C}})\\
\end{tikzcd}
\]
The top and the bottom face commute because $\mathcal{F}_{0}$ satisfies the unitality condition. The back face commutes as $\mathcal{F}_{0}$ is natural and due to the definition of the monoidal product in Arr($\mathcal{D}$). Moreover, the left face commutes because of the definition of the unitor in $\mathcal{C}$ and the right face commutes since $\mathcal{F}_{2}$ is natural. Finally, the front face also commutes due to the definition of the unitor in $\mathcal{C}$.
\end{proof}
 
Similarly, a \emph{braided} monoidal functor $\mathcal{F}: \mathcal{C} \longrightarrow \mathcal{D}$ gives rise to a braided monoidal functor between Arr($\mathcal{C}$) and Arr($\mathcal{D}$):
\begin{satz}\label{prop. braided monodial functor in arrow cats}
Let $(\mathcal{C}, \otimes_{\mathcal{C}}, \mathbb{I}_{\mathcal{C}}, \sigma_{\mathcal{C}})$ and $(\mathcal{D}, \otimes_{\mathcal{D}}, \mathbb{I}_{\mathcal{D}}, \sigma_{\mathcal{D}})$ be two braided monoidal categories and $(\mathcal{F}, \mathcal{F}_{2}, \mathcal{F}_{0}) $ be a braided monoidal functor between them. Then $(\mathcal{F}, \mathcal{F}_{2}, \mathcal{F}_{0}) $ gives rise to a braided monoidal functor $(\Tilde{\mathcal{F}}, \Tilde{\mathcal{F}}_{2}, \Tilde{\mathcal{F}}_{0})$ between Arr($\mathcal{C}$) and Arr($\mathcal{D}$).
\end{satz}
\begin{proof}
We only need to show that the following diagram commutes:
\[\begin{tikzcd}[row sep=1cm, column sep=1cm, inner sep=2ex]
& \mathcal{F}(V)\otimes_{\mathcal{D}} \mathcal{F}(V') \arrow[swap]{dl}{\mathcal{F}_{2, V, V'}}\arrow{rr}{\sigma_{\mathcal{F}(V), \mathcal{F}(V')}} \arrow[dotted]{dd}[yshift=-0.5cm]{\mathcal{F}(f)\otimes_{\mathcal{D}} \mathcal{F}(f')} & & \mathcal{F}(V')\otimes_{\mathcal{D}} \mathcal{F}(V) \arrow[swap]{dl}{\mathcal{F}_{2, V, V'}} \arrow[swap]{dd}[yshift=-0.5cm]{\mathcal{F}(f')\otimes_{\mathcal{D}} \mathcal{F}(f)} \\
\mathcal{F}(V\otimes_{\mathcal{C}} V') \arrow{rr}[xshift= 1cm]{\mathcal{F}(\sigma_{V, V'})} \arrow{dd} {\mathcal{F}(f\otimes_{\mathcal{C}}f')} & & \mathcal{F}(V' \otimes_{\mathcal{C}} V) \arrow{dd} [yshift=0.5cm]{\mathcal{F}(f'\otimes_{\mathcal{C}} f)}\\
& \mathcal{F}(W)\otimes_{\mathcal{D}} \mathcal{F}(W') \arrow[dotted]{dl}{\mathcal{F}_{2, W, W'}} \arrow[dotted]{rr}[xshift= -1cm]{\sigma_{\mathcal{F}(W), \mathcal{F}(W')}} & & \mathcal{F}(W')\otimes_{\mathcal{D}} \mathcal{F}(W) \arrow{dl}{\mathcal{F}_{2, W', W}} \\
\mathcal{F}(W \otimes_{\mathcal{C}}W') \arrow[swap]{rr}{\mathcal{F}(\sigma_{W, W'})} & & \mathcal{F}(W'\otimes_{\mathcal{C}} W)\\
\end{tikzcd}
\]
The top and the bottom face commute because $\mathcal{F}$ is a braided monoidal functor and the two side faces commute due to naturality of $\mathcal{F}_{2}$. The front face commutes due to naturality of the braiding in $\mathcal{C}$. Finally, the back face commutes per definition of the braiding and hence the whole diagram commutes.
\end{proof}
\begin{satz}\label{re. symmetric monoidal functor between arrow cats}
  If the functor $\mathcal{F}: \mathcal{C} \longrightarrow \mathcal{D}$ in the above construction is symmetric, i. e. if $\mathcal{C}$ is symmetric, then so is $\Tilde{\mathcal{F}}: \text{Arr}(\mathcal{C}) \longrightarrow \text{Arr}(\mathcal{D})$.  
\end{satz}
It is straightforward to verify that the following statement is true:
\begin{example} Let $\mathbb{K}$ be a field. An $n$-dimensional topological quantum field theory is given by a symmteric monoidal functor from the cobordism category to the category of vector spaces (see~\cite{Baez_1995}):
\begin{align}
   \mathcal{Z}: \mathbf{Cob}_{n} \longrightarrow \mathbf{Vect}_{\mathbb{K}}
\end{align}
This functor assigns to each closed oriented $(n-1)$-dimensional manifold $M$ a $\mathbb{K}$-vector space $\mathcal{Z}(M)$ and to each oriented bordism $B$ from an $(n-1)$-dimensional manifold $M$ to another $(n-1)$-dimensional manifold $N$ a $\mathbb{K}$-linear map $\mathcal{Z}(B): \mathcal{Z}(M) \longrightarrow \mathcal{Z}(N)$. \\\\
Applying  Prop.~\ref{re. symmetric monoidal functor between arrow cats}, we can construct a symmetric monoidal functor
\begin{align}
    \Tilde{\mathcal{Z}}: \mathrm{Arr}(\mathbf{Cob}_{n}) \longrightarrow \mathrm{Arr}(\mathbf{Vect}_{\mathbb{K}})
\end{align}
which assigns to each oriented bordism $B$ from an $(n-1)$-dimensional manifold $M$ to another $(n-1)$-dimensional manifold $N$ a $\mathbb{K}$-linear map $\mathcal{Z}(B): \mathcal{Z}(M) \longrightarrow \mathcal{Z}(N)$ and to each morphism $(\beta_{M}, \beta_{N}): B \longrightarrow B'$
\[\begin{tikzcd}
    M \arrow{r}{\beta_{M}} \arrow[swap]{d}{B} &  \arrow{d}{B'} M' \\
   N \arrow{r}{\beta_{N}} & N'
\end{tikzcd}
\]
where $\beta_M: M  \longrightarrow M'$ and $\beta_{N}: N \longrightarrow N'$ are bordisms, a morphism $(\mathcal{Z}(\beta_{M}), \mathcal{Z}(\beta_{N})): \mathcal{Z}(B) \longrightarrow \mathcal{Z}(B')$
\[\begin{tikzcd}
   \mathcal{Z}(M) \arrow{r}{ \mathcal{Z}(\beta_{M})} \arrow[swap]{d}{ \mathcal{Z}(B)} &  \arrow{d}{ \mathcal{Z}(B')}  \mathcal{Z}(M') \\
    \mathcal{Z}(N) \arrow{r}{ \mathcal{Z}(\beta_{N})} &  \mathcal{Z}(N')
\end{tikzcd}
\]
where $\mathcal{Z}(\beta_{M}): \mathcal{Z}(M) \longrightarrow \mathcal{Z}(M')$ and $\mathcal{Z}(\beta_{N}): \mathcal{Z}(N) \longrightarrow \mathcal{Z}(N')$ are $\mathbb{K}$-linear maps.
\end{example}
\begin{satz}\label{prop. monoidal natural trafo between arrow cats}
Let $(\mathcal{C}, \otimes_{\mathcal{C}}, \mathbb{I}_{\mathcal{C}})$ and $(\mathcal{D}, \otimes_{\mathcal{D}}, \mathbb{I}_{\mathcal{D}}) $ be two monoidal categories and $(\mathcal{F}, \mathcal{F}_{2}, \mathcal{F}_{0}) $ and $(\mathcal{G}, \mathcal{G}_{2}, \mathcal{G}_{0}) $ be two monoidal functors between those categories. A monoidal natural transformation $\eta: \mathcal{F} \Longrightarrow \mathcal{G}$ gives rise to a monoidal natural transformation $\Tilde{\eta}: \Tilde{\mathcal{F}} \Longrightarrow \Tilde{\mathcal{G}}$, where $\Tilde{\mathcal{F}}$ and $\Tilde{\mathcal{G}}$ are the functors defined on the arrow categories.
\end{satz}
\begin{proof} We only need to show that $\Tilde{\eta}$ preserves the monoidal product and the monoidal unit in Arr($\mathcal{C}$), if $\eta$ preserves the monoidal product and the monoidal unit in $\mathcal{C}$. For this, the following diagram has to commute:
\[\begin{tikzcd}[row sep=1cm, column sep=1cm, inner sep=2ex]
& \mathcal{F}(A)\otimes_{\mathcal{D}} \mathcal{F}(A) \arrow[swap]{dl}{\mathcal{F}_{2, A, C}}\arrow{rr}{\eta_{A}\otimes_{\mathcal{D}}\eta_{B}} \arrow[dotted]{dd}[yshift=-0.5cm]{\mathcal{F}(f)\otimes_{\mathcal{D}} \mathcal{F}(g)} & & \mathcal{G}(A)\otimes_{\mathcal{D}} \mathcal{G}(C) \arrow[swap]{dl}{\mathcal{G}_{2, A, C}} \arrow[swap]{dd}[yshift=-0.5cm]{\mathcal{G}(f)\otimes_{\mathcal{D}} \mathcal{G}(g)} \\
\mathcal{F}(A\otimes_{\mathcal{C}} C) \arrow{rr}[xshift= 1cm]{\eta_{A\otimes_{\mathcal{C}} C}} \arrow{dd} {\mathcal{F}(f\otimes_{\mathcal{C}}g)} & & \mathcal{G}(A \otimes_{\mathcal{C}} C) \arrow[crossing over]{dd} [yshift=0.5cm]{\mathcal{G}(f\otimes_{\mathcal{C}} g)}\\
& \mathcal{F}(B)\otimes_{\mathcal{D}} \mathcal{F}(D) \arrow[dotted]{dl}{\mathcal{F}_{2, B, D}} \arrow[dotted] {rr}[xshift= -1cm]{\eta_{B} \otimes_{\mathcal{D}} \eta_{D}} & & \mathcal{G}(B)\otimes_{\mathcal{D}} \mathcal{G}(D) \arrow{dl}{\mathcal{G}_{2, B, D}} \\
\mathcal{F}(B \otimes_{\mathcal{C}}D) \arrow[swap]{rr}{\eta_{B\otimes_{\mathcal{C}} D}} & & \mathcal{G}(B\otimes_{\mathcal{C}} D)\\
\end{tikzcd}
\]
The top and the bottom face commute because $\eta$ is a monoidal natural transformation and the two side faces commute due to naturality of $\mathcal{F}_{2}$ resp. $\mathcal{G}_{2}$. Finally, the front and the back face commute due to naturality of $\eta$ and thus the whole diagram commutes. Moreover, the following diagram has to commute as well:
\[\begin{tikzcd}[row sep=scriptsize, column sep=scriptsize]
\mathcal{F}(\mathbb{I}_{\mathcal{C}}) \arrow{dd}{\mathcal{F}(\mathrm{id}_{\mathbb{I}_{\mathcal{C}}})} \arrow{dr}{\mathcal{F}_{0}} \arrow{rr}{\eta_{\mathbb{I}_{\mathcal{C}}}} & &  \mathcal{G}(\mathbb{I}_{\mathcal{C}}) \arrow{dd}{\mathcal{G}(\mathrm{id}_{\mathbb{I}_{\mathcal{C}}})} \arrow{dl}{\mathcal{G}_{0} } \\
& \mathbb{I}_{\mathcal{D}} \arrow{dd}[yshift= 0.5 cm]{\mathrm{id}_{\mathbb{I}_{\mathcal{D}}}} & \\
\mathcal{F}(\mathbb{I}_{\mathcal{C}}) \arrow{dr}{\mathcal{F}_{0}} \arrow[dotted]{rr}[xshift= -0.5cm]{\eta_{\mathbb{I}_{\mathcal{C}}}} & & \mathcal{G}(\mathbb{I}_{\mathcal{C}}) \arrow{dl}{\mathcal{G}_{0}} \\
 & \mathbb{I}_{\mathcal{D}}  & 
\end{tikzcd} 
\]
Here the top and the bottom face commute because $\eta$ is a monoidal natural transformation and the two side faces commute due to the naturality of $\mathcal{F}_{0}$ resp. $\mathcal{G}_{0}$. The back face also commutes as $\eta$ is a natural and thus the whole diagram commutes.
\end{proof} 
\subsection{Duals and Pivots in Arrow Categories} \label{sec: duals in arrow cats}
In this section we will investigate under which circumstances an arrow category is rigid. As it turns out, we can, similarly to the constructions made earlier, use the structure of $\mathcal{C}$ to define a evaluation and coevaluation map on the subcategory of Arr($\mathcal{C}$) where all objects are isomorphisms, i.e. where all objects are morphisms in the \emph{core} of $\C$. Furthermore, we will show that a pivot in a pivotal category $\mathcal{C}$ induces a pivot in its arrow category and that a twist in a ribbon category gives rise to a twist in its arrow category.
\begin{theorem}\label{thrm. duals in arrow cats}
Let $(\mathcal{C}, \otimes_{\mathcal{C}}, \mathbb{I}_{\mathcal{C}})$ be a rigid monoidal category. An object $f: A \longrightarrow B$ in Arr($\mathcal{C}$) has a dual if and only if it is an isomorphism in $\mathcal{C}$.

\end{theorem}
\begin{proof}
Let $d_{A}: A^{*}\otimes_{\mathcal{C}}A \longrightarrow \mathbb{I}$ and $b_{A}: \mathbb{I} \longrightarrow A \otimes_{\mathcal{C}} A^{*}$, where $A, A^{*}$ $\in$ obj($\mathcal{C}$) respectively, define the evaluation and coevaluation map in $\mathcal{C}$. For a morphism $f: A\longrightarrow B$ $\in$ obj(Arr($\mathcal{C}$)) and its dual $f^{*}: A^{*} \longrightarrow B^{*}$ we can now define maps in Arr($\mathcal{C}$) :
 \begin{align}
     d_{f}: f^{*} \otimes_{\mathrm{Arr}(\mathcal{C})} f \longrightarrow \mathrm{id}_{\mathbb{I}_{\mathcal{C}}}\label{death}
 \end{align}
 and 
  \begin{align}
     b_{f}: \mathrm{id}_{\mathbb{I}_{\mathcal{C}}}  \longrightarrow  f^{*} \otimes_{\mathrm{Arr}(\mathcal{C})} f, \label{birth}
 \end{align}
such that 
 \[\begin{tikzcd} 
    A^{*} \otimes_{\mathcal{C}} A \arrow{r}{d_{A}} \arrow[swap]{d}{f^{*}\otimes_{\mathcal{C}} f} &  \arrow{d}{ \mathrm{id}_{\mathbb{I}_{\mathcal{C}}}} \mathbb{I}_{\mathcal{C}} \\
    B^{*}\otimes_{\mathcal{C}} B \arrow{r}{d_{B}} & \mathbb{I}_{\mathcal{C}}
\end{tikzcd}
\]
and
\[\begin{tikzcd} 
    \mathbb{I}_{\mathcal{C}} \arrow{r}{b_{A}} \arrow[swap]{d}{\mathrm{id}_{\mathbb{I}_{\mathcal{C}}}} &  \arrow{d}{ f\otimes_{\mathcal{C}} f^{*}} A \otimes_{\mathcal{C}} A^{*} \\
   \mathbb{I}_{\mathcal{C}}  \arrow{r}{b_{B}} & B\otimes_{\mathcal{C}} B^{*}
\end{tikzcd}
\]
commute. The first snake identity in Arr($\mathcal{C}$) is given by:
 \[\begin{tikzcd}
    A  \arrow{rr}{\mathrm{id}_{A}}\arrow[swap]{dd}{f} \arrow[swap]{dr} [xshift=0.5cm]{\mathrm{id}_{A}\otimes_{\mathcal{C}} b_{A}} & &  \arrow[swap]{dd}{f} A \\
    & A \otimes_{\mathcal{C}} (A\otimes_{\mathcal{C}} A^{*}) \arrow{ur}{\mathrm{id}_{A} \otimes_{\mathcal{C}}d_{A}} \arrow{dd}[yshift=0.5cm]{f\otimes_{\mathcal{C}}(f\otimes_{\mathcal{C}}f^{*})} & \\
    B \arrow{dr}[swap]{\mathrm{id}_{B} \otimes_{\mathcal{C}}b_{B}} \arrow[dotted]{rr}[xshift=-0.5cm]{\mathrm{id}_{B}} & &   B\\
    & \arrow{ur}{\mathrm{id}_{B} \otimes_{\mathcal{C}}d_{B}} B \otimes_{\mathcal{C}} ( B\otimes_{\mathcal{C}}B^{*}) &\\
   \end{tikzcd}
\]
The top and the bottom face commute because $b$ and $d$ fulfil the snake identities in $\mathcal{C}$ for the dualities $A\dashv A^{*}$ and $B\dashv B^{*}$ respectively. The two side faces commute due to the definition of the monoidal product and the (co)evaluation map in Arr($\mathcal{C}$). Finally, it is easy to see that the back face also commutes and thus the whole diagram commutes. Analogously, one can show that the second snake identity is satisfied in Arr($\mathcal{C}$).\\\\ 
In terms of the graphical calculus, commutativity of the diagram corresponding to Eq.~\ref{death} means that the following equation has to hold: 
\begin{figure}[H]
    \centering
    \includegraphics[scale=0.7]{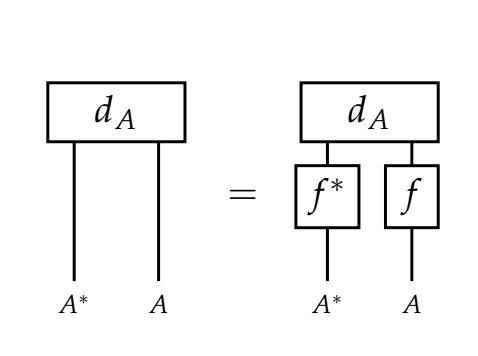}
\end{figure}
Inserting this into the snake identity in $\mathcal{C}$, yields: 
\begin{figure}[H]
    \centering
    \includegraphics[scale=0.6]{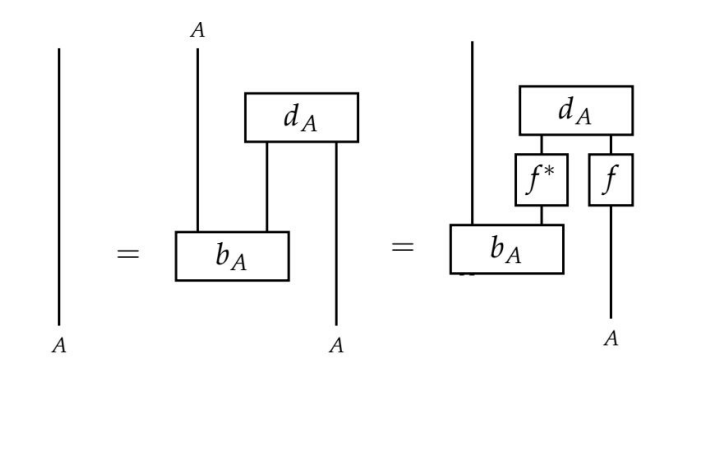}
\end{figure}
For the other snake identity we get a similar expression. From this we can conclude that $f$ has to be an isomorphism with inverse:
\begin{figure}[H]
    \centering
    \includegraphics[scale=0.7]{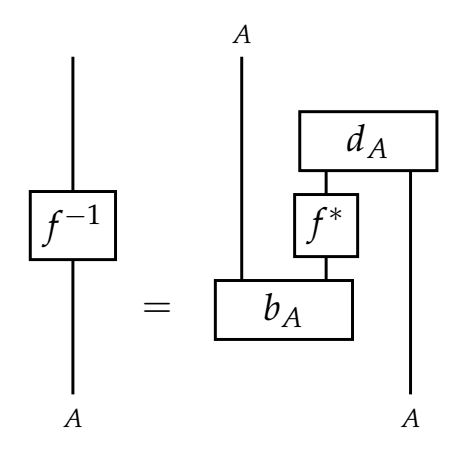}
\end{figure}
On the other hand, if we assume that $f$ is invertible with inverse given by the morphism in the diagram above, it is easy to verify that the diagrams corresponding to Eq.~\ref{birth} and Eq.~\ref{death} and hence also the snake identities in Arr($\C$) commute. From that one can conclude that an object $f$ in Arr($\mathcal{C}$) has a dual if and only if it is invertible.
\end{proof}
\begin{example}\label{ex. duals in Arr(Rel)}
In \textbf{Mat}($\mathbb{N}$), only invertible matrices have duals.
\end{example}\noindent
In the following we will denote the core of an arbitrary category $\C$ with $\C_{core}$.
\begin{theorem}\label{thrm. pivots in arrow cats}
Let $\mathcal{C}$ be a pivotal category with pivot $\pi_{A}: A \longrightarrow A^{**}$ for $A$ $\in$ obj($\mathcal{C}$). Then Arr($\mathcal{C}_{core}$) is also pivotal with pivot $\Tilde{\pi}_{f}=(\pi_{A}, \pi_{B}): f \longrightarrow f^{**}$ for $f: A \longrightarrow B$ $\in$ obj(Arr($\mathcal{C}$)).
\end{theorem} 
\begin{proof}
Since the pivot is a monoidal natural transformation between the identity functor and the functor that sends each object to its double dual, we can apply Proposition~\ref{prop. monoidal functors in arrow categories} and Proposition~\ref{prop. natural trafo between arrow cats} to define a monoidal natural transformation between the identity functor on Arr($\mathcal{C}_{core}$) and the functor that sends each object in Arr($\mathcal{C}_{core}$) to its double dual.
\end{proof}

\begin{example}
The arrow category of the subcategory of $\mathbf{Mat}(\N)$ where all matrices are invertible is a pivotal symmetric monoidal category.
\end{example}
\begin{satz}\label{prop. ribbon arrow cats}
Let  $(\mathcal{C}, \otimes_{\mathcal{C}}, \mathbb{I}_{\mathcal{C}})$ be a ribbon category, i. e. a rigid monoidal category with a twist $\theta_{A}: A \longrightarrow A$. Then $Arr(\mathcal{C}_{core})$ is a ribbon category with twist $(\theta_{A}, \theta_{B}): f \longrightarrow f$, i. e. we have:
\[\begin{tikzcd}
 A  \arrow{r}{\theta_{A}} \arrow[swap]{d}{f} & \arrow{d}{f} A  \\
   B   \arrow{r}{\theta_{B}} &  B 
   \end{tikzcd}\]
\end{satz}
\begin{proof}
 In order to define a twist $(\theta_{A}, \theta_{B}): f \longrightarrow f$, the following diagram has to commute:
\[\begin{tikzcd}[row sep=1cm, column sep=1cm, inner sep=2ex]
& C \otimes_{\mathcal{C}} A
\arrow[swap]{dl}{\sigma_{A,C}^{-1}}\arrow{rr}{\sigma_{C,A}} \arrow[dotted]{dd}[yshift=-0.5cm]{g \otimes_{\mathcal{C}} f} & & A \otimes_{\mathcal{C}} C \arrow[swap]{dl}{\theta_{A} \otimes_{\mathcal{C}} \theta_{C}} \arrow{dd}{f \otimes_{\mathcal{C}} g} \\
A \otimes_{\mathcal{C}} C \arrow [crossing over]{rr}[xshift= 0.5cm]{\theta_{A \otimes C}} \arrow{dd} {f \otimes_{\mathcal{C}} g} & & A \otimes_{\mathcal{C}} C \arrow[crossing over]{dd} [yshift=0.5cm]{f \otimes_{\mathcal{C}} g}\\
& D \otimes_{\mathcal{C}} B \arrow[dotted]{dl}{\sigma_{B,D}^{-1}} \arrow[dotted] {rr}[xshift= -0.5cm]{\sigma_{D,B}} & & B \otimes_{\mathcal{C}} D \arrow{dl}{\theta_{B} \otimes_{\mathcal{C}} \theta_{D}} \\
B \otimes_{\mathcal{C}} D \arrow[swap]{rr}{\theta_{B\otimes D}} & & B \otimes_{\mathcal{C}} D\\
\end{tikzcd}
\]
Here the bottom and the top face commute because $\theta$ is a twist in $\mathcal{C}$. Moreover, the front face commutes due to naturality of $\theta$ in $\mathcal{C}$. The left and the back face commute due to naturality of the braiding and its inverse and the right face commutes because of the definition of the monoidal product and naturality of the braiding.\\\\
The second identity that has to be fulfilled is given by
\begin{align*}
    (\theta_{\mathbb{I}_{\mathcal{C}}}, \theta_{\mathbb{I}_{\mathcal{C}}}) = (\mathrm{id}_{\mathbb{I}_{\mathcal{C}}}, \mathrm{id}_{\mathbb{I}_{\mathcal{C}}}) 
\end{align*}
This equation holds because $\theta$ is a twist in $\mathcal{C}$ and hence: $\theta_{\mathbb{I}_{\mathcal{C}}} =\mathrm{id}_{\mathbb{I}_{\mathcal{C}}}$. The last condition that has to be satisfied is given by:
\begin{align*}
    (\theta_{B^{*}}, \theta_{A^{*}})= (\theta_{A}, \theta_{B})^{*} = (\theta_{B}^{*}, \theta_{A}^{*}) 
\end{align*}
which holds since $\theta$ is a twist and hence we have: $(\theta_{B}^{*}, \theta_{A}^{*}) = (\theta_{B^{*}}, \theta_{A^{*}})$.
\end{proof}
\begin{example}
Consider again the arrow category of the subcategory of $\mathbf{Mat}(\N)$ where all matrices are invertible. The pivot induces a twist on this category via $\theta_{A}=\mathrm{id}_{A}$ and hence this category is a ribbon category. In particular, it is a compact category.
\end{example}
\subsection{(Co)monoids, Bialgebras, Frobenius Structures and Hopf Algebras in Arrow Categories} \label{sec: monoids comonoids in arr cats}
In this section we will show that the (co)monoids in arrow categories are given by (co)monoid-morphisms. With that, we can then derive a notion of bialgebras, Frobenius algebras and Hopf algebras in arrow categories. In the following we will use $\eta$ to denote the unit of a monoid in a category. This should not be confused with a natural transformation.
\begin{theorem}\label{thrm. monoids in arrow cats}
Let $\mathcal{C}$ be a monoidal category and ($A$, $\mu_{A}$, $\eta_{A}$) and ($B$, $\mu_{B}$, $\eta_{B}$) be monoids in $\mathcal{C}$. Then a morphism of monoids  $f: A \longrightarrow B$ in $\mathcal{C}$ is a monoid in Arr($\mathcal{C}$). If the monoids $A$ and $B$ are commutative, so is the monoid $f: A\longrightarrow B $ in Arr($\mathcal{C}$).
\end{theorem}
\begin{proof}
Let $\mathcal{C}$ be a monoidal category and ($A$, $\mu_{A}$, $\eta_{A}$) and ($B$, $\mu_{B}$, $\eta_{B})$ be monoids in $\mathcal{C}$. We can then define a monoid object $f: A \longrightarrow B $ in Arr($\mathcal{C}$), where $f$ is a morphism of monoids, via the following construction:
\begin{itemize}
    \item the multiplication is given by a morphism $\Tilde{\mu}=(\mu_{A}, \mu_{B})$ such that
    \[\begin{tikzcd}
 A \otimes_{\mathcal{C}} A  \arrow{r}{\mu_{A}} \arrow[swap]{d}{f \otimes_{\mathcal{C}}f} & \arrow{d}{f} A  \\
   B \otimes_{\mathcal{C}}B  \arrow{r}{\mu_{B}} & B
   \end{tikzcd}\]
   commutes.
    \item the unitor is a morphism $\Tilde{\eta} =(\eta_{A}, \eta_{B})$ such that
    \[\begin{tikzcd}
  \mathbb{I}_{\mathcal{C}} \arrow{r}{\eta_{A}} \arrow[swap]{d}{\mathrm{id}_{\mathbb{I}_{\mathcal{C}}}} & \arrow{d}{f} A  \\
   \mathbb{I}_{\mathcal{C}} \arrow{r}{\eta_{B}} & B
   \end{tikzcd}\]
   commutes.
\end{itemize}
We still need to show that the pentagon axiom and the unitor diagrams commute. This will be done in Appendix~\ref{app. proofs}.

\end{proof}
 Analogously, one can show that the following statement is true:
\begin{theorem}\label{rem.comonoids in arrow cats}
 Given a category $\C$, then a comonoid morphism in $\mathcal{C}$ is a comonoid in Arr($\mathcal{C}$). If the comonoid in $\mathcal{C}$ is cocommutative, then this also holds for the comonoid in Arr($\mathcal{C}$).     
\end{theorem}
Having defined both monoid and comonoid objects in a Arr($\mathcal{C}$), we can now define bialgebra objects in Arr($\mathcal{C}$). In the following, $\Delta$ denotes the comultiplication and $\epsilon$ is the counit 
of a comonoid in $\C$.
\begin{theorem}\label{thrm. bialgebras in arrow cats}
Let $(\mathcal{C}, \otimes_{\mathcal{C}}, \mathbb{I}_{\mathcal{C}})$ be a monoidal category and ($A$, $\mu_{A}$, $\eta_{A}$, $\Delta_{A}$, $\epsilon_{A}$) and ($B$, $\mu_{B}$, $\eta_{B}$, $\Delta_{B}$, $\epsilon_{B}$) be both a monoid and a comonoid in $\mathcal{C}$ such that the bialgebra axiom is satisfied. Then a morphim $f:A \longrightarrow B$ in $\mathcal{C}$, which is both a monoid and comonoid morphism, is a bialgebra object in Arr($\mathcal{C}$).
\end{theorem} 
 \begin{proof}
Let ($A$, $\mu_{A}$, $\eta_{A}$, $\Delta_{A}$, $\epsilon_{A}$) and ($B$, $\mu_{B}$, $\eta_{B}$, $\Delta_{B}$, $\epsilon_{B}$) be two objects which are both a monoid and a comonoid and both satisfy the bialgebra axiom.
We then have that each $f: A\longrightarrow B$ which is both a monoid and a comonoid morphism is a monoid and a comonoid in Arr($\mathcal{C}$). The morphism $f$ is a bialgebra object in Arr($\mathcal{C}$), if it satisfies the two bialgebra axioms in Arr($\mathcal{C}$). The first axiom requires that the following diagram commutes:
 \[\begin{tikzcd}[row sep=1cm, column sep=1cm, inner sep=2ex]
& A \otimes A
\arrow[swap]{dl}{\Delta_{A} \otimes \Delta_{A}}\arrow{rrrr}{\mu_{A}} \arrow[dotted]{dd}[yshift=-0.5cm]{f \otimes f} & & & & A  \arrow[swap]{dl}{\Delta_{A}} \arrow{dd}{f} \\
(A \otimes A) \otimes (A \otimes A) \arrow [crossing over]{rr}[xshift= 1cm]{\mathrm{id}_{A} \otimes \sigma_{A,A} \otimes \mathrm{id}_{A}} \arrow{dd} {(f \otimes f) \otimes (f \otimes f)} & & (A \otimes A) \otimes (A \otimes A) \arrow [yshift=0.5cm]{dd} {(f \otimes f) \otimes (f \otimes f)} \arrow{rr}{\mu_{A} \otimes \mu_{A}} & &  A \otimes A \arrow [yshift= 0.5 cm]{dd}{f \otimes f} \\
& B \otimes B \arrow[dotted]{dl}{\Delta_{B} \otimes \Delta_{B}} \arrow[dotted] {rrrr}[xshift= -0.5cm]{\mu_{B}} & & & & B \arrow{dl}{\Delta_{B}} \\
(B \otimes B) \otimes (B \otimes B) \arrow[swap]{rr}{\mathrm{id}_{B} \otimes \sigma_{B,B} \otimes \mathrm{id}_{B}} & & (B \otimes B) \otimes (B \otimes B) \arrow{rr}{\mu_{B} \otimes \mu_{B}} & & B\otimes B \\
\end{tikzcd}
\]
The left face commutes because $f$ is a morphism of comonoids and because of the definition of the monoidal product in Arr($\mathcal{C}$). Analogously, the right face commutes since $f$ is a morphism of comonoids. Because $A$ and $B$ are bialgebra objects in $\mathcal{C}$, the top and the bottom face commute. The back face and the right square in the front commute because $f$ is a monoid morphism and because of the definition of the monoidal product in Arr($\mathcal{C}$). Finally, the left square in the front also commutes because of the definition of the braiding in $\mathcal{C}$ and the monoidal product in Arr($\mathcal{C}$) and thus the whole diagram commutes.\\\\
The second bialgebra axiom requires that the following diagram commutes:
\[\begin{tikzcd}[row sep=1cm, column sep=2cm, inner sep=2ex]
& A \otimes A
\arrow[swap]{dl}{\mu_{A} }\arrow{rr}{\epsilon_{A}\otimes \epsilon_{A}} \arrow[dotted]{dd}[yshift=-0.5cm]{f \otimes f} & & \mathbb{I} \otimes \mathbb{I} \arrow[swap]{dl}{\Tilde{=}} \arrow{dd}{\mathrm{id}_{\mathbb{I}}\otimes \mathrm{id}_{\mathbb{I}}} \\
A \arrow [crossing over]{rr}[xshift= 0.5cm]{\epsilon_{A}} \arrow{dd} {f} & & \mathbb{I} \arrow [crossing over]{dd} [yshift=0.5cm]{\mathrm{id}_{\mathbb{I}}}\\
& B \otimes B \arrow[dotted]{dl}{\mu_{B}} \arrow[dotted] {rr}[xshift= -0.5cm]{\epsilon_{B}\otimes \epsilon_{B}} & & \mathbb{I} \otimes \mathbb{I} \arrow{dl}{\Tilde{=}} \\
B \arrow[swap]{rr}{\epsilon_{B}} & & \mathbb{I} \\
\end{tikzcd}
\]
Here the back face commutes because of the definition of the monoidal product in Arr($\mathcal{C}$) and because $f$ is a morphism of comonoids. The left and the front face also commute, because $f$ is a morphism of monoids. Finally, the top and bottom face commute as $A$ and $B$ are bialgebra objects in $\mathcal{C}$ and thus the whole diagram commutes.
\end{proof}

\begin{satz}\label{prop. FA in arrow cat}
Let $(\mathcal{C}, \otimes_{\mathcal{C}}, \mathbb{I}_{\mathcal{C}})$ be a monoidal category and $A$, $B$ be two Frobenius structures in  $\mathcal{C}$ (in particular, $A$ and $B$ are both monoids and comonoids). Then an object $f: A\longrightarrow B$ in Arr($\mathcal{C}$), which is both a monoid and a comonoid morphism is a Frobenius structure in Arr($\mathcal{C}$). If $A$ and $B$ are special Frobenius algebras in $\mathcal{C}$, then $f: A\longrightarrow B$ is a special Frobenius structure in Arr($\mathcal{C}$).
\end{satz}
\begin{proof}
The Frobenius law in Arr($\mathcal{C}$) is given by the following diagram:
\[\begin{tikzcd}[row sep=1cm, column sep=2cm, inner sep=2ex]
& A \otimes A
\arrow[swap]{dl}{\mathrm{id}_{A} \otimes \Delta_{A}}\arrow{rr}{\Delta_{A} \otimes \mathrm{id}_{A}} \arrow[dotted]{dd}[yshift=-0.5cm]{f \otimes f} & & A \otimes A \otimes A \arrow[swap]{dl}{\mathrm{id}_{A} \otimes \mu_{A}} \arrow{dd}{f\otimes f \otimes f} \\
A \otimes A \otimes A  \arrow [crossing over]{rr}[xshift= 1cm]{\mu_{A} \otimes \mathrm{id}_{A} } \arrow{dd} {f \otimes f \otimes f} & & A \otimes A \arrow [crossing over]{dd} [yshift=0.5cm]{f \otimes f}\\
& B \otimes B \arrow[dotted]{dl}{\mathrm{id}_{B} \otimes \Delta_{B}} \arrow[dotted] {rr}[xshift= -1cm]{\Delta_{B} \otimes \mathrm{id}_{B}} & & B \otimes B \otimes B \arrow{dl}{\mathrm{id}_{B} \otimes \mu_{B}} \\
B \otimes B \otimes B  \arrow[swap]{rr}{\mu_{B} \otimes \mathrm{id}_{B}} & & B \otimes B \\
\end{tikzcd}
\]
The top and the bottom face of the diagram commutes because $A$ and $B$ are Frobenius algebras in $\mathcal{C}$. The left face commutes because of the definition of the monoidal product in Arr($\mathcal{C}$) and because $f$ is a morphism of comonoids. Analogously, the back face commutes. Finally, the front and the right face commute because of the definition of the monoidal product in $\mathcal{C}$ and because $f$ is a morphism of monoids and thus the whole diagram commutes.\\\\
If $A$ and $B$ are both special, then the following diagram also commutes as $f$ is both a morphism of monoids and comonoids:
\[\begin{tikzcd}[row sep=1cm, column sep=2cm, inner sep=2ex]
 A \arrow[swap]{dr}{ \Delta_{A}}\arrow{rr}{\mathrm{id}_{A}} \arrow[swap]{dd}[yshift=-0.5cm]{f} & &  A  \arrow{dd}{f} \\
& A \otimes A \arrow[swap]{ur}{\mu_{A}} \arrow{dd}[yshift= 0.5cm] {f \otimes f } & \\
  B \arrow[swap]{dr}{ \Delta_{B}}\arrow[dotted]{rr} [xshift= -0.5cm] {\mathrm{id}_{B}}  & &  B \\
& B \otimes B \arrow[swap]{ur}{\mu_{B}} & \\
\end{tikzcd}
\]
and hence $f$ is also special.
\end{proof}
\begin{satz}\label{prop. dagger Frobenius Algebras in aaro cat}
Let $(\mathcal{C}, \otimes_{\mathcal{C}}, \mathbb{I}_{\mathcal{C}})$ be a monoidal dagger category and $A$, $B$ be two objects in  $\mathcal{C}$ which are dagger Frobenius algebras in $\mathcal{C}$. Then a unitary morphism of both monoids and comonoids $f:A \longrightarrow B$ is a dagger Frobenius structure in Arr($\mathcal{C}$).
\begin{proof}
Let $f:A \longrightarrow B$ be a morphism of bialgebras, i. e. a moprhism of monoids and comonoids that is also unitary, i. e.: $f^{\dagger} = f^{-1}$. We can define a multiplication and a comultiplication in Arr($\mathcal{C}$) via 
\begin{align*}
    \mu_{f} &= (\mu_{A}, \mu_{B}): f\otimes_{\mathcal{C}} f \longrightarrow f \ \ \ \ \text{and} \\
    \Delta_{f} &= (\Delta_{A}, \Delta_{B}): f \longrightarrow f\otimes_{\mathcal{C}} f 
\end{align*}
respectively. Applying the dagger functor from Prop.~\ref{prop. dagger functor} on the comultiplication then gives
\begin{align*}
    \Delta_{f}^{\dagger}= (\Delta_{A}, \Delta_{B})^{\dagger}= (\Delta_{A}^{\dagger}, \Delta_{B}^{\dagger})= (\mu_{A}, \mu_{B}) = \mu_{f}
\end{align*}
where we used the fact that $A$ and $B$ are dagger Frobenius algebras. Similarly, we get that
\begin{align*}
    \epsilon_{f}^{\dagger}= (\epsilon_{A}^{\dagger}, \epsilon_{B}^{\dagger}) = (\eta_{A}, \eta_{B})= \eta_{f}
\end{align*}
and hence $f:A \longrightarrow B$ is a dagger Frobenius structure in Arr($\mathcal{C}$).
\end{proof}
\end{satz}
This leads us to the following theorem:
\begin{theorem}
Let $(\mathcal{C}, \otimes_{\mathcal{C}}, \mathbb{I}_{\mathcal{C}})$ be a braided monoidal dagger category and let $A$ and $B$ be two objects in  $\mathcal{C}$ which are commutative special dagger Frobenius algebras in $\mathcal{C}$. Then a unitary morphism of Frobenius algebras $f: A \longrightarrow B$ is a commutative special dagger Frobenius structure in Arr($\mathcal{C}$).
\end{theorem}
\begin{proof}
We only need to apply Theorem~\ref{prop. monoidal product in arrow cats}, Proposition~\ref{prop. FA in arrow cat} and Proposition~\ref{prop. dagger Frobenius Algebras in aaro cat}.
\end{proof}
\begin{theorem}\label{thr. Hopf alegbras in arrow cats}
Let $\mathcal{C}$ be a monoidal category and ($A$, $\mu_{A}$, $\eta_{A}$, $\Delta_{A}$, $\epsilon_{A}$, $S_{A}$) and ($B$, $\mu_{B}$, $\eta_{B}$, $\Delta_{B}$, $\epsilon_{B}$, $S_{B}$) be Hopf algebra objects in $\mathcal{C}$ where $S_{A}: A \longrightarrow A$ and $S_{B}: B \longrightarrow B$ are antipodes in $\mathcal{C}$. Then morphisms of Hopf algebras are Hopf algebra objects in Arr($\mathcal{C}$).
\end{theorem} 
\begin{proof}
Let $\mathcal{C}$ be a monoidal category and ($A$, $\mu_{A}$, $\eta_{A}$, $\Delta_{A}$, $\epsilon_{A}$, $S_{A}$) and ($B$, $\mu_{B}$, $\eta_{B}$, $\Delta_{B}$, $\epsilon_{B}$, $S_{B}$) be Hopf algebra objects in $\mathcal{C}$, where $S_{A}: A \longrightarrow A$ and $S_{B}: B \longrightarrow B$ are antipodes in $\mathcal{C}$. We can then define an antipode  $S=(S_{A}, S_{B})$
\[\begin{tikzcd}
  A \arrow{r}{S_{A}} \arrow[swap]{d}{f} & \arrow{d}{f} A  \\
   B \arrow{r}{S_{B}} & B
   \end{tikzcd}\]
where $f$ is a morphism of antipodes and thus a Hopf algebra object in Arr($\mathcal{C}$).

We now have to prove that the Hopf algebra axiom is satisfied if $f$ is a morphism of Hopf algebras. In Arr($\mathcal{C}$) this means that the following diagram has to commute:
\[\begin{tikzcd}[row sep=1cm, column sep=2cm, inner sep=2ex]
& A \arrow[swap]{dl}{\Delta_{A}}\arrow{rrrr}{\epsilon_{A}} \arrow[dotted]{dd}[yshift=-0.5cm]{f} & & & & \mathbb{I}  \arrow[swap]{dl}{\eta_{A}} \arrow{dd}{\mathrm{id}_{\mathbb{I}}} \\
A \otimes A \arrow [crossing over]{rr}[xshift= 1cm]{\mathrm{id}_{A} \otimes S_{A}} \arrow{dd} [yshift=0.5cm] {f \otimes f} & & A \otimes A \arrow[crossing over]{dd} [yshift= 0.5cm]{f \otimes f}\arrow{rr}{\mu_{A}} & & \arrow{dd}[yshift=0.5cm]{f} A \\
& B  \arrow[dotted]{dl}{\Delta_{B}} \arrow[dotted] {rrrr}[xshift= -0.5cm]{\epsilon_{B}} & & & &  \mathbb{I} \arrow{dl}{\eta_{B}} \\
B \otimes B \arrow[swap]{rr}{\mathrm{id}_{B} \otimes S_{B}} & &  B \otimes B  \arrow{rr}{\mu_{B}} & & B\\
\end{tikzcd}
\]

Because $f$ is a morphism of comonoids, the left and the right side face commute. Analogously, the second square of the front face and the back face commute, since $f$ is a morphism of monoids. The first square of the front face commutes because of the definition of the monoidal product in Arr($\mathcal{C}$) and the definition of the antipode in Arr($\mathcal{C}$). Finally, the top and the bottom face commute as $A$ and $B$ are Hopf algebras in $\mathcal{C}$ and hence the whole diagram commutes. 
\end{proof}
\appendix \label{app. proofs}
\section{Proofs}

\paragraph{Proof of \ref{prop. monoidal product in arrow cats} continued:}
\begin{proof}
We want to show that the pentagon and the triangle axiom are satisfied for the monoidal product defined in the beginning of this proof. The pentagon axiom in Arr($\mathcal{C}$) is given by the diagram on the next page where the front and the back face commute because $\alpha$ satisfies the pentagon axiom ($\alpha$ is the associator in $\C$). The two side faces commute due to the definition of the monoidal product and naturality of the associator in $\C$. The two top faces and the bottom face also commute due to naturality of the associator in $\C$ and hence the whole diagram commutes.\footnote{We have abbreviated the labels of the monoidal products and the associators trying to make the diagram more readable.}
\[
\def\stimes{{\otimes}}
\hspace{-3cm}
\small
\begin{tikzcd}[ampersand replacement=\&]
\& A_{1}\otimes (A_{2} \otimes (A_{3} \otimes A_{4})) \arrow[swap]{dl}{f_{1} \otimes ( f_{2} \otimes (f_{3} \otimes f_{4}))}\arrow{r}{\alpha} \arrow[dotted]{dd}[yshift=0.75cm]{\mathrm{id}_{A_{1}} \otimes \alpha} \&  (A_{1}\otimes A_{2}) \otimes (A_{3} \otimes A_{4}) \arrow{dl}{(f_{1} \otimes f_{2}) \otimes (f_{3} \otimes f_{4})} \arrow{r}{\alpha}  \&  ((A_{1}\otimes A_{2}) \otimes A_{3}) \otimes A_{4} \arrow{dl}{((f_{1} \otimes f_{2}) \otimes f_{3} )\otimes f_{4}} \arrow{dd}{\alpha \otimes \mathrm{id}_{A_{4}}}
\\
B_{1}\otimes (B_{2} \otimes (B_{3} \otimes B_{4})) \arrow {r}{\alpha} \arrow{dd} {\mathrm{id}_{B_{1}} \otimes \alpha} \& (B_{1}\otimes B_{2}) \otimes (B_{3} \otimes B_{4})\arrow{r}{\alpha} \&   ((B_{1}\otimes B_{2}) \otimes B_{3}) \otimes B_{4} \arrow[crossing over]{dd}[yshift=1cm]{\alpha}\\
\& A_{1}\otimes ((A_{2} \otimes A_{3}) \otimes A_{4})\arrow[dotted]{dl}{f_{1} \otimes ((f_{2}\otimes f_{3} )\otimes f_{4})} \arrow[dotted]{rr}[xshift=1cm]{\alpha} \& \& (A_{1}\otimes (A_{2} \otimes A_{3})) \otimes A_{4} \arrow{dl}{(f_{1} \otimes (f_{2} \otimes f_{3} ))\otimes f_{4}} \\
B_{1}\otimes ((B_{2} \otimes B_{3}) \otimes B_{4}) \arrow[swap]{rr}{\alpha} \& \& (B_{1}\otimes (B_{2} \otimes B_{3})) \otimes B_{4}
\end{tikzcd}
\]
The triangle axiom is given by the diagram: 
\[\begin{tikzcd}[row sep=scriptsize, column sep=scriptsize]
(A\otimes \mathbb{I}) \otimes A' \arrow{dd}{(f\otimes \mathrm{id}_{\mathbb{I}}) \otimes f'} \arrow{dr}{\rho \otimes \mathrm{id}_{A'} } \arrow{rr}{\alpha} & &  A\otimes (\mathbb{I} \otimes A') \arrow{dd}{f\otimes (\mathrm{id}_{\mathbb{I}} \otimes f')} \arrow{dl}{\mathrm{id}_{A}\otimes \lambda } \\
&   A \otimes A' \arrow{dd}[yshift= 0.5 cm]{f\otimes f'} & \\
(B\otimes \mathbb{I}) \otimes B'\arrow{dr}{\rho \otimes \mathrm{id}_{B'} } \arrow[dotted]{rr}[xshift= 1cm]{\alpha} & & B\otimes (\mathbb{I} \otimes B') \arrow{dl}{\mathrm{id}_{B}\otimes \lambda} \\
 & B \otimes B'  & 
\end{tikzcd} 
\]
Here the top and the bottom face commute because $\rho$ and $\alpha$ satisfy the triangle identity in $\C$ and the two side faces commute due to naturality of the left and right unitors in $\mathcal{C}$. Finally, the back face commutes because of the naturality of the associator in $\C$.    
\end{proof}
\paragraph{Proof of Theorem~\ref{thrm. monoids in arrow cats} continued:}
\begin{proof}
We will first show that the multiplication in the above construction satisfies the pentagon axiom. For an arbitrary morphism $f: A \longrightarrow B$ between two monoid objects $A$, $B$ $\in$ obj($\mathcal{C}$) the pentagon axiom in $\C$ gives rise to the following diagram:
\[
\begin{tikzcd} [row sep=2 cm, column sep=2cm, inner sep=1ex]
& (A\otimes A) \otimes A \arrow[swap]{dl}{(f \otimes f) \otimes f}\arrow{r}{\alpha} \arrow[dotted]{dd}[yshift=1cm]{\mu_{A} \otimes \mathrm{id}_{A}} &  A\otimes (A \otimes A) \arrow{dl}{f \otimes (f \otimes f)} \arrow{r}{\mathrm{id}_{A} \otimes \mu_{A}}  &   A\otimes A \arrow{dl}{f\otimes f} \arrow{dd}{\mu_{A}}
\\
(B\otimes B) \otimes B \arrow {r}{\alpha} \arrow{dd} [yshift= 1cm] {\mu_{B} \otimes \mathrm{id}_{B}} & [crossing over]B\otimes (B \otimes B)\arrow{r}{\mathrm{id}_{B} \otimes \mu_{B}} &  B \otimes B \arrow[crossing over]{dd}[yshift=1cm]{\mu_{B}}\\
&  A\otimes A \arrow[dotted]{dl}{f \otimes f } \arrow[dotted]{rr}[xshift=1cm]{\mu_{A}} & & A \arrow{dl}{f} \\
B \otimes B \arrow[swap]{rr}{\mu_{B}} & & B\\
\end{tikzcd}
\]
The front and the back face commute because $A$ and $B$ are monoids in $\mathcal{C}$. Moreover, the bottom and the right face commute because $f$ is a morphism of monoids. The left face commutes because of the definition of the monoidal product in Arr($\mathcal{C}$) and because $f$ is a morphism of monoids. The left square in the top commutes because of naturality the associator in $\mathcal{C}$ and the right square in the top commutes because of the definition of the monoidal product in Arr($\mathcal{C}$) and because $f$ is a morphism of monoids. Hence the whole diagram commutes and the pentagon axiom is satisfied.\\\\
Moreover, we need to show that the unitor diagram commutes:
\[\begin{tikzcd}[row sep=2cm, column sep=2cm, inner sep=2ex]
\mathbb{I} \otimes_{\mathcal{C}} A \arrow{dd}{\mathrm{id}_{\mathbb{I}}\otimes_{\mathcal{C}} f} \arrow{dr}{\lambda_{A}} \arrow{rr}{\eta_{A}\otimes_{\mathcal{C}} \mathrm{id}_{A}} & &  A\otimes_{\mathcal{C}} A \arrow[dotted]{dd}{f \otimes_{\mathcal{C}} f} \arrow{dl}{\mu_{A}} & & \arrow{ll}{\mathrm{id}_{A} \otimes_{\mathcal{C}} \eta_{A}} \arrow{dd}{f \otimes_{\mathcal{C}} \mathrm{id}_{\mathbb{I}}} \arrow[crossing over]{dlll}{\rho_{A}} A \otimes_{\mathcal{C}} \mathbb{I} \\
&  A \arrow{dd}[yshift= 1cm]{f} & & &\\
\mathbb{I} \otimes_{\mathcal{C}} B  \arrow{dr}{\lambda_{B}} \arrow[dotted]{rr}[xshift=1cm]{\eta_{B}\otimes_{\mathcal{C}} \mathrm{id}_{B}} & &  B\otimes_{\mathcal{C}} B  \arrow[dotted]{dl}{\mu_{B}} & & \arrow[dotted]{ll}{\mathrm{id}_{B} \otimes_{\mathcal{C}} \eta_{B}}  \arrow{dlll}{\rho_{B}} B \otimes_{\mathcal{C}} \mathbb{I} \\
 &   B & & &
\end{tikzcd} 
\]
The top and the bottom face commute because $A$ and $B$ are monoids in $\mathcal{C}$. The left and the right face commute due to the definition of left and right unitors in $\mathcal{C}$ and the definition of the monoidal product in Arr($\mathcal{C}$). Because $f$ is a morphism of monoids, the inner face also commutes. The two squares in the back commute as well, because of the definition of the monoidal product in Arr($\mathcal{C}$) and because $f$ is a monoid morphism. Hence the whole diagram commutes.
\\\\
Finally, commutativity means that the following diagram has to commute
\[\begin{tikzcd}[row sep=2cm, column sep=2cm]
A \otimes_{\mathcal{C}} A \arrow{dd}{f \otimes_{\mathcal{C}} f} \arrow{dr}{\mu_{A}} \arrow{rr}{\sigma_{A,A}} & &  A\otimes_{\mathcal{C}} A \arrow{dd}{f \otimes_{\mathcal{C}} f} \arrow{dl}{\mu_{A} } \\
& A \arrow{dd}[yshift= 0.5 cm]{f} & \\
B \otimes_{\mathcal{C}} B  \arrow{dr}{\mu_{B}} \arrow[dotted]{rr}[xshift=-1cm]{\sigma_{B,B}} & &  B\otimes_{\mathcal{C}} B \arrow{dl}{\mu_{B} } \\
 & B  & 
\end{tikzcd} 
\]
which holds if both $A$ and $B$ are commutative monoids in $\mathcal{C}$.
\end{proof}


\begin{thebibliography}{1}

\bibitem{Baez_1995}
John~C. Baez and James Dolan.
\newblock Higher-dimensional algebra and topological quantum field theory.
\newblock {\em Journal of Mathematical Physics}, 36(11):6073--6105, nov 1995.

\bibitem{jamie}
C.~Heunen J.~Vicary.
\newblock {\em Categories for Quantum Theory}.
\newblock Oxford University Press, 2020.

\bibitem{IntroCat}
Steven Roman.
\newblock {\em An Introduction to the Language of Category Theory}.
\newblock Birkhäuser, 2017.

\bibitem{white2018arrow}
David White and Donald Yau.
\newblock Arrow categories of monoidal model categories, 2018.

\bibitem{graphicalcalculus}
Christopher~J. Wood, Jacob~D. Biamonte, and David~G. Cory.
\newblock Tensor networks and graphical calculus for open quantum systems.
\newblock 2011.

\end{thebibliography}
\end{document}